\documentclass[12pt]{amsart}
\usepackage{amsmath, amssymb, amsfonts}
\usepackage[pdftex]{graphicx}
\usepackage{verbatim} 
\usepackage{tikz}
\usepackage{tikz-cd}

\theoremstyle{plain}
\newtheorem{theorem}{Theorem}[section]
\newtheorem{lemma}[theorem]{Lemma}
\newtheorem{corollary}[theorem]{Corollary}
\newtheorem{prop}[theorem]{Proposition}

\newtheorem{conj}[theorem]{Conjecture}

\theoremstyle{remark}

\newtheorem{remark}[theorem]{Remark}

\newtheorem*{note*}{Note}
\newtheorem*{remark*}{Remark}
\newtheorem*{example*}{Example}

\theoremstyle{definition}
\newtheorem*{definition*}{Definition}
\newtheorem*{hypothesis*}{Hypothesis}
\newtheorem*{assumptions*}{Assumptions}
\newtheorem{definition}[theorem]{Definition}
\newtheorem{notation}[theorem]{Notation}

\newcommand{\Z}{\mathbb{Z}}

\newcommand{\F}{\mathbb{F}}

\newcommand{\Aut}{\mathrm{Aut}}

\newcommand{\GL}{\mathrm{GL}}

\newcommand{\End}{\mathrm{End}}

\numberwithin{equation}{section}

\newcommand{\Hol}{\mathrm{Hol}}
\newcommand{\Mat}{\mathrm{Mat}}
\newcommand{\Perm}{\mathrm{Perm}}

\title[quaternion and dihedral braces and Hopf--Galois structures]{On the number of quaternion and dihedral braces and Hopf--Galois structures}

\author{Nigel P. Byott}
\address[]{\parbox{\linewidth}{Department of Mathematics \& Statistics, University of Exeter, North Park Road, Exeter EX4 4QF, United Kingdom\vspace{0.15cm}}}
\email{n.p.byott@exeter.ac.uk}
\urladdr{https://mathematics.exeter.ac.uk/staff/NPByott}

\author{Fabio Ferri}
\address[]{\parbox{\linewidth}{131 Hayward Road, Bristol BS5 9PY, United Kingdom\vspace{0.15cm}}}
\email{fabioferri94@gmail.com}

\thanks{
Byott was partially supported by the Engineering and Physical Sciences Research Council [grant number EP/V005995/1]. 
Ferri was supported by the Engineering and Physical Sciences Research Council [grant number EP/W52265X/1].
\newline
For the purpose of open access, the authors have applied a CC BY public copyright licence  to any Author Accepted Manuscript version arising.
\newline
Data Access Statement: No data was used for the research described in this article.
}  


\date{version of \today}

\begin{document}

\maketitle

\begin{abstract}
We prove a conjecture of Guarnieri and Vendramin on the number of braces of a given order whose multiplicative group is a generalised quaternion group. At the same time, we give a similar result where the multiplicative group is dihedral. We also enumerate Hopf-Galois structures of abelian type on Galois extensions with generalised quaternion or dihedral Galois group.
 \end{abstract}

\section{Introduction}
Braces were introduced by Rump \cite{MR2278047} as a generalisation of radical rings in order to study set-theoretical solutions of the Yang-Baxter Equation (YBE). Every brace can be viewed as a non-degenerate involutive solution of the YBE, and any such solution can be embedded in a brace. Braces have since been studied from a number of different perspectives, and generalised in a variety of ways. Of particular relevance to this paper is the connection with Hopf-Galois theory, first noted by Bachiller in \cite{MR3465351}.

As reformulated in \cite{MR3177933}, the definition of a brace may be given as follows.
\begin{definition}
A left brace $(B,+,\circ)$ consists of a set $B$ and two binary operations $+$, $\circ$ such that
\begin{itemize}
	\item[(i)] $(B,+)$ is an abelian group;
	\item[(ii)] $(B,\circ)$ is a group;
	\item[(iii)] the brace relation $a \circ (b+c) = a \circ b - a + a \circ c$ holds for all $a$, $b$, $c \in B$.
\end{itemize}
The left brace $(B,+,\circ)$ is {\em trivial} if the operations $+$, $\circ$ coincide. 
\end{definition}

There is an analogous notion of right brace, where (iii) is replaced by $(a+b) \circ c = a\circ c -c +b \circ c$. In this paper, we only consider left braces, so we will omit the adjective ``left".

Guarnieri and Vendramin \cite{MR3647970} generalised braces to skew braces by removing the requirement that the group $(B,+)$ be abelian. In the same paper, they presented the results of computer calculations counting braces (and skew braces) of small order. On the basis of these, they formulated a number of conjectures, including the following:
\begin{conj} \cite[Conjecture 6.6]{MR3647970} \label{quat-conj}
Let $m \geq 3$ and let $q(4m)$ be the number of braces $B$ whose multiplicative group $(B,\circ)$ is a generalised quaternion group of order $4m$. Then 
$$ q(4m) = \begin{cases} 2 & \mbox{ if } m \mbox{ is odd,} \\
	                     7 & \mbox{ if } m \equiv 0 \pmod{8}, \\
	                     9 & \mbox{ if } m \equiv 4 \pmod{8}, \\
	                     6 & \mbox{ if } m \equiv 2 \pmod{8} \mbox{ or } m \equiv 6 \pmod{8}.
	       \end{cases} $$
\end{conj}
A partial proof of this conjecture was given by Rump \cite{MR4141382}, who showed that $q(2^n)=7$ for all $n \geq 5$ using the equivalence between braces and affine structures of groups.  
We note that some other enumerative conjectures from \cite{MR3647970} have been proved in \cite{MR3765444}, \cite{MR3962864} and \cite{MR4297324}. For further computational results counting braces and skew braces, see \cite{MR3974481}, \cite{MR4113853}, \cite{MR4223285}. Based on their extensive computations in \cite{MR4113853}, Bardakov, Neshchadim and Yadav made conjectures on the number of braces of order $8p$ and $12p$ for large enough primes $p$. These have recently been proved by Crespo, Gil-Mu\~{n}oz, Rio and Vela \cite{MR4559373, MR4513787}, and we will use similar techniques in this paper; cf.~Remark \ref{rem:CGRV}.

In this paper, we will prove Conjecture \ref{quat-conj} in full (see Theorem \ref{thm:GV-conj}). Our methods are rather different from those used by Rump in \cite{MR4141382}. We will also answer Questions 6.7 and 6.8 of \cite{MR3647970} which concern the additive groups $(B,+)$ for these braces (see Table \ref{table:abgps}). At the same time, we give corresponding results for braces where $(B,\circ)$ is a dihedral group. We also give enumerative results on the Hopf-Galois structures related to these braces. More precisely, given a Galois extension of fields  $L/K$ whose Galois group is a generalised quaternion or dihedral group of a given order, we count the number of Hopf-Galois structures on $L/K$ of each possible abelian type. 	       

The connection between (skew) braces and Hopf-Galois theory, which is fundamental to our approach in this paper, comes about because both can be described in terms of regular subgroups in the holomorph of a group $N$. A subgroup $M$ in the group $\Perm(X)$ of permutations of a set $X$ is said to be regular if, given any $x$, $y \in X$, there is a unique $\pi \in M$ with $\pi(x)=y$. The conjugate of a regular subgroup by any element of $\Perm(X)$ is again regular.  For any group $N$, we define the holomorph $\Hol(N)$ of $N$ to be $N \rtimes \Aut(N)$, and we view
$\Hol(N)$ as a group of permutations of $N$ (so that the normal subgroup $N$ of $\Hol(N)$ corresponds to left translations). We next explain how regular subgroups are related to skew braces and to Hopf-Galois structures.

On the one hand, if $(B,+,\circ)$ is a (skew) brace, then there is a homomorphism of groups $(B,\circ) \to \Aut(B,+)$ given by $a \mapsto \lambda_a$, where $\lambda_a(b)=-a + a \circ b$. The map $a \mapsto (a, \lambda_a)$ embeds $(B,\circ)$ as a regular subgroup of $\Hol(B,+)$. Conversely, every regular subgroup $G$ of the holomorph of a group $N$ gives rise to a skew brace $(B,+,\circ)$ with additive group $N$ and multiplicative group $G$ \cite[\S4]{MR3647970}. Two regular subgroups $G_1$ and $G_2$ of $\Hol(N)$ correspond to isomorphic skew braces if and only if they are conjugate by an element of $\Aut(N)$ \cite[Proposition A.3]{MR3763907}. By \cite[Lemma 2.1]{MR4113853}, $G_1$ and $G_2$ are conjugate by an element of $\Aut(N)$ if and only if they are conjugate by an element of $\Hol(N)$, so the isomorphism classes of skew braces with additive group $N$ correspond to conjugacy classes of regular subgroups in $\Hol(N)$. 

On the other hand, let $L/K$ be a finite Galois extension of fields with Galois group $G$. Greither and Pareigis \cite{MR878476} showed that Hopf-Galois structures on $L/K$ correspond to regular subgroups $N$ in $\Perm(G)$ which are normalised by left translations by elements of $G$. We will call the isomorphism class of $N$ the {\em type} of the corresponding Hopf-Galois structure. Starting with an abstract group $N$, one can determine the Hopf-Galois structures of type $N$ on $L/K$ from the regular subgroups of $\Hol(N)$ which are isomorphic to $G$. More precisely, if $e'(G,N)$ is the number of these regular subgroups, and $e(G,N)$ is the number of Hopf-Galois structures of type $N$ on $L/K$, then from  \cite{MR1402555} we have the formula
\begin{equation} \label{HGS-count} 
		e(G,N) = \frac{|\Aut(G)|}{|\Aut(N)|}  \, e'(G,N) .
\end{equation}
For further background on Hopf-Galois structures, see \cite[Chapter 2]{MR1767499}.
 
In this paper we will be concerned with the case when $N$ is abelian and $G$ is either quaternion or dihedral.
We will determine the abelian groups $N$ so that $\Hol(N)$ contains a regular quaternion or dihedral subgroup $G$, and we will count the number of such subgroups for each such $N$. This is the first step both to determining  
the number of quaternion or dihedral braces with additive group isomorphic to $N$, and to determining the number of Hopf-Galois structures of type $N$ on a Galois extension of fields with quaternion or dihedral Galois group.
There are two, essentially distinct, parts in our calculation of these numbers. Firstly, we consider the case when $|G|=|N|=2^n$ for some $n \geq 2$. Our method here can be regarded as an extension of the ideas introduced in the thesis of Featherstonhaugh \cite{MR2704825}, and further developed in \cite{MR2901229}, which show that if $N$ is a finite abelian $p$-group of rank $m<p-1$ then any regular abelian subgroup $G$ of $\Hol(N)$ is isomorphic to $N$. In the setting of braces, Bachiller \cite[Theorem 2.5]{MR3465351} removed the requirement that $G$ be abelian, showing that the number of elements of any given order is the same in $G$ and $N$. Using a description of $\Aut(N)$ due to Hillar and Rhea \cite{MR2363058}, we obtain a similar but weaker relation between the structures of $G$ and $N$ when the hypothesis on $m$ is removed (Lemma \ref{lem:boundHol}). In particular, we show that if $|N|=2^n$ and $\Hol(N)$ contains a regular subgroup $G$ which is a dihedral or quaternion group then $N$ contains a cyclic subgroup of index at most $8$ (Theorem \ref{thm:types}). We can then find the numbers of braces and Hopf-Galois structures in the $2$-power case by examining each of the possible groups $N$ in turn. Secondly, we show that the general case, $|N|=2^n s$ for $s$ odd, can essentially be reduced to the $2$-power case. The numbers of braces and Hopf-Galois structures are independent of $s$ when $s\geq 3$. (The number of braces differs from that for $s=1$ in some cases.)  

In some of our calculations, we use the computer algebra package \textsc{Magma}. All the counts we give could in principle be obtained by purely theoretical considerations, but to do so would have required the examination of a number of special cases and would have significantly increased the length of the paper.

Finally, we emphasize that (despite the occurrence of Conjecture \ref{quat-conj} in the paper \cite{MR3647970} where skew braces were introduced) in this paper we are considering braces rather than skew braces. Thus the additive group $N \cong (B,+)$ is always assumed to be abelian.  
	
\section{The holomorph of a finite abelian $p$-group}
In this section, we recall from \cite{MR2363058} a description of the endomorphism ring of a finite abelian $p$-group $N$, and then use this to bound the order of elements of $p$-power order in $\Hol(N)$ in terms of the rank and exponent of $N$. Although only the case $p=2$ is needed later in the paper, in this section we allow $p$ to be an arbitrary prime.

\begin{prop}\label{pro:automorphismsmatrix}
Let 
 \[
 N=\Z/p^{a_1}\Z\times \Z/p^{a_2}\Z\times \cdots \times \Z/p^{a_r}\Z
\]
be a finite abelian $p$-group, with $1\leq a_1\leq\cdots\leq a_r$.

For $1 \leq i, j \leq r$, let 
$$  R_{ij} =  p^{\max\{0,a_i-a_j\}}(\Z/p^{a_r}\Z), $$
and let 
$$ R:=(R_{ij})_{1\leq i,j\leq r},  $$ 
the set of matrices whose $(i,j)$-entry is in $R_{ij}$. Then $R$ 
has a ring structure with sum and multiplication inherited from the ring $\Mat_r(\Z)$ of $r \times r$ matrices over $\Z$, and there exists a surjective ring homomorphism 
\[
 f: R \twoheadrightarrow \End(N).
\] 
Moreover,
\[
 \ker f=(p^{a_i}(\Z/p^{a_r}\Z))_{i,j}\subseteq R.
\]

Finally, an element of $\End(N)$ represented by $(r_{ij})_{i,j}\in R$ is in $\Aut(N)$ if and only if the reduction of $(r_{ij})_{i,j}$ in $\Mat_r(\F_p)$ is invertible.
\end{prop}

\begin{proof}
By looking at the image of the elements of the canonical basis, one can verify that an element of $R$ acts as an endomorphism of $N$ via its matrix representation, and that each endomorphism has a representative in $R$. A complete proof can be found in \cite[Theorem 3.3]{MR2363058}. In \cite[Lemma 3.4]{MR2363058} the authors show, using adjoint matrices, that $\ker f$ consists of the matrices whose $(i,j)$-entry is divisible by $p^{a_i}$. In \cite[Theorem 3.6]{MR2363058} they show that the elements of $\Aut(N)$ are those corresponding to the invertible matrices modulo $p$. It is worth remarking that all the proofs are straightforward, although they involve several necessary steps to be verified. 
\end{proof}

Proposition \ref{pro:automorphismsmatrix} means that we can identify $\End(N)$ with the ring
$$ \overline{R} = R/\ker f $$
of $r \times r$ matrices whose $i$th row consists of elements of $\Z/p^{a_i}\Z$, with the condition that the $(i,j)$-entry is divisible by $p^{a_i-a_j}$ when $a_i>a_j$. (This condition ensures that multiplication of matrices in $\overline{R}$ is well-defined.) We correspondingly view elements of $N$ as column vectors with $i$th entry in $\Z/p^{a_i}\Z$. Then $\Aut(N)$ is identified with the group $\overline{R}^\times$ of units in $\overline{R}$, which consists of elements of $\overline{R}$ whose reduction in $\Mat_r(\F_p)$ lies in $\GL_r(\F_p)$. 

\begin{corollary}\label{cor:holomorphmatrix}
 With the hypotheses and notation of Proposition \ref{pro:automorphismsmatrix}, an element of $\Hol(N)$ can be represented as
 \[
 \begin{pmatrix}
    A       & v  \\
    0       & 1
\end{pmatrix},
\]
where $A$ is a matrix in $\overline{R}$ whose reduction is in $\GL_r(\F_p)$, and where $v\in N$. (We view the last row as having entries in $\Z/p^{a_r}\Z$.)
\end{corollary}

We next describe a particular Sylow $p$-subgroup in $\overline{R}^\times$.

\begin{lemma} \label{lem:Sylow}
Let $U$ be the subgroup of unipotent upper triangular matrices in $\GL_r(\F_p)$:
$$ (b_{ij}) \in U \Leftrightarrow b_{ii}=1 \mbox{ for all } i, \mbox{ and } b_{ij}=0 \mbox{ if } i>j.  $$ 
Let $P$ be the subgroup of $\overline{R}^\times$ consisting of matrices whose reduction in $\GL_r(\F_p)$ lies in $U$. Then $P$ is a Sylow $p$-subgroup of $\overline{R}^\times$. 
\end{lemma}
\begin{proof}
Let $W$ be the image of $R^\times$ in $\GL_r(\F_p)$. Then $W$  consists of those invertible matrices $(b_{ij})$ with $b_{ij}=0$ if $a_i>a_j$. In particular, $U \subseteq W$. It is well-known that $U$ is a Sylow $p$-subgroup of $\GL_r(\F_p)$, so it is also a Sylow $p$-subgroup of $W$. The kernel $K$ of the reduction map $\overline{R} \to \Mat_r(\F_p)$ clearly has $p$-power order. By Proposition \ref{pro:automorphismsmatrix}, if $A$ is an element of $R$ whose reduction is invertible, then $A$ is invertible in $R$. Thus the kernel of the reduction map on units $\overline{R}^\times \to \GL_r(\F_p)$ is $1+K$. Since $1+K$ is a $p$-group, it follows that the preimage of $U$ in $\overline{R}^\times$ is a Sylow $p$-subgroup of $\overline{R}^\times$.
\end{proof}

We now seek to relate possible orders of elements of $\Hol(N)$ to the exponent of $N$. This is done in Lemma \ref{lem:boundHol}. We first give two preliminary results.

\begin{prop}\label{pro:matrixF2}
	Let $m \geq 1$ and let $t=\lceil\log_p m\rceil$. Then, for any $X\in\GL_m(\F_p)$ of $p$-power order, we have $X^{p^t}=I$.
\end{prop}
\begin{proof}
Replacing $X$ by a conjugate, we may assume that $X$ lies in the Sylow $p$-subgroup $U$ of $\GL_m(\F_p)$ consisting of unipotent upper triangular matrices. By induction on $0 \leq k \leq m$, we easily see that the $(i,j)$-entry of $(X-I)^k$ is $0$ unless $i+j>m+k$. In particular,  $(X-I)^m=0$. Since $p^t \geq m$ we have $X^{p^t}-I=(X-I)^{p^t}=0$.
\end{proof}

\begin{prop}\label{pro:induction1}
	Let $m \geq 1$, let $t=\lceil\log_p m\rceil$, and let $d \geq 1$.
	Then, for any matrix $X$ of $p$-power (multiplicative) order in $\Mat_m(\Z/p^d\Z)$, we have $X^{p^{t+d-1}}=I$.
\end{prop}
\begin{proof}
We argue by induction on $d$. The case $d=1$ is Proposition \ref{pro:matrixF2}. 
Now suppose that $d>1$ and the assertion holds for $d-1$. Let 
$X \in \Mat_m(\Z/p^d\Z)$ have $p$-power order. Reducing mod $p^{d-1}$ and applying the induction hypothesis, we have 
$$  X^{p^c}-I \in p^{d-1}  \Mat_m(\Z/p^d\Z) $$
where $c=t+(d-1)-1$. Now 
$$ X^{p^{c+1}}-I = (X^{p^c}-I) (X^{(p-1)p^c} + X^{(p-2)p^c} + \cdots + X^{p^c}+I). $$
Since $d-1 \geq 1$, we have $X^{p^c} \equiv I \pmod{p \, \Mat_m(\Z/p^d\Z)}$, so that
$$ X^{(p-1)p^c} + X^{(p-2)p^c} + \cdots + X^{p^c}+I \equiv pI \equiv 0
  \pmod{p \, \Mat_m(\Z/p^d\Z)}, $$
  and hence
$$ X^{p^{t+d-1}}-I = X^{p^{c+1}}-I \in (p^{d-1}  \Mat_m(\Z/p^d\Z)) (p \, \Mat_m(\Z/p^d\Z)) = \{0\}. $$ 
This completes the induction. 
\end{proof}

\begin{lemma}\label{lem:boundHol}
Let $N$ be a finite abelian $p$-group of rank $r$ and exponent $p^d$. If $\Hol(N)$ contains an element of order $p^k$ then $k< \lceil \log_p(r+1)\rceil +d$. 
\end{lemma}
\begin{proof}
We may represent $N$ as in Proposition \ref{pro:automorphismsmatrix} with $a_r=d$. Let $m=r+1$. Then, by Corollary \ref{cor:holomorphmatrix}, each element of $\Hol(N)$ can be represented as an invertible $m \times m$ matrix whose $i$th row consists of elements of $\Z/p^{a_i}\Z$. Let $X$ be such a matrix representing an element of order $p^k$, and let $\widehat{X}$ be any matrix in $\Mat_m(\Z/p^d\Z)$ which reduces to $X$. Then $\widehat{X}^{p^{k-1}} \neq I$ in $\Mat_m(\Z/p^d\Z)$, so by Proposition \ref{pro:induction1} we have $k-1< \lceil \log_p(r+1)\rceil +d-1$.
\end{proof}

In the setting of Lemma \ref{lem:boundHol}, if $\Hol(N)$ contains an element of order $p^{d+1}$ then $r\geq p$. Thus if $r\leq p-1$ then the exponent of a $p$-subgroup of $\Hol(N)$ is at most $p^d$. For related results on braces and Hopf-Galois structures in the case $r\leq p-1$, see \cite{MR3465351} and \cite{MR2901229} respectively. Our application of Lemma \ref{lem:boundHol} will be for $p=2$ and $r>p$.

\section{Preliminary results on quaternion and dihedral braces}
\begin{notation} \label{not:def-groups}
 Let $n\geq 2$ be an integer and let $s$ be an odd number. We will denote a quaternion, respectively dihedral, group $G$ by $\langle x,y\rangle_o$ if
 \[
  G=Q_{2^ns}=\langle x,y:x^{2^{n-1}s}=1,yx=x^{-1}y,y^2=x^{2^{n-2}s} \rangle,
 \]
respectively
 \[
  G=D_{2^ns}=\langle x,y:x^{2^{n-1}s}=1,yx=x^{-1}y,y^2=1 \rangle
 \]
where we mean that the order of $x$ in $G$ is exactly $2^{n-1}s$, (note that this is stronger than simply saying that $G$ is generated by $x$ and $y$ as we are prescribing their roles as elements of the group). We include the degenerate case $n=2$ and $s=1$, so that $Q_4\cong C_4$ and $D_4\cong C_2\times C_2$.
\end{notation}

\begin{lemma}\label{lem:conjugationy}
 Let $n\geq 4$ be an integer and let $G_1=\langle x_1,y_1\rangle_o$ and $G_2=\langle x_2,y_2\rangle_o$ be two groups isomorphic to either $Q_{2^n}$ or $D_{2^n}$. Let $f:G_1\rightarrow G_2$ be an isomorphism. Then $f(x_1)$ is an odd power of $x_2$.
\end{lemma}

\begin{proof}
 This follows from the fact that all the elements of the type $x_i^ay_i$ have order $2$ or $4$, so that $f$ can only send $\langle x_1\rangle$ to $\langle x_2\rangle$.
\end{proof}

\begin{prop}\label{pro:automorphismsquaterniondihedral}
 Let $G$ be a finite group of order $2^n$, which is either quaternion with $3\neq n\geq 2$ or dihedral with $n\geq 3$. Then the number of automorphisms of $G$ is $2^{2n-3}$.
\end{prop}

\begin{proof}
 See \cite{MR2363137} just before Lemma 8.1 for the complete proof. We sketch the proof for the convenience of the reader. If $G=\langle x,y\rangle_o$ and $f\in\Aut(G)$, then $f(x)$ can only be an odd power of $x$ and $f(y)$ can only be $x^iy$ for a certain $i$. We can verify that every such choice defines an automorphism.
\end{proof}

We now apply Lemma \ref{lem:boundHol} to find the possible types of quaternion and dihedral braces of $2$-power order.

\begin{theorem}  \label{thm:types}
Let $n\geq 2$ be an integer. Let $N$ be an abelian group of order $2^n$ with exponent $2^d$. Suppose that there is a regular quaternion or dihedral subgroup of $\Hol(N)$. Then $N$ must be one of the following groups:
\begin{itemize}
	\item $C_{2^n}$ for $n\geq 2$;
	\item $C_2\times C_{2^{n-1}}$ for $n\geq 2$;
	\item $C_4\times C_{2^{n-2}}$ for $n\geq 3$;
	\item $C_2\times C_2\times C_{2^{n-2}}$ for $n\geq 3$;
	\item $C_2\times C_2\times C_2\times C_{2^{n-3}}$ for $n\geq 4$.
\end{itemize}
\end{theorem}
\begin{proof}
Let $r$ be the rank of $N$. Since $\Hol(N)$ must contain an element of order $2^{n-1}$, Lemma \ref{lem:boundHol} gives 
$n-1 < \lceil \log_2(r+1)\rceil + d$. Also, as $N$ has at least one cyclic factor of order $2^d$, we have $r \leq n-d+1$. Hence
$$ r-1 \leq n-d < \lceil \log_2(r+1)\rceil + 1. $$
Thus $r \leq 4$. Moreover, if $r=2$ then $n-d=1$ or $2$; if $r=3$ then $n-d=2$; if $r=4$ then $n-d=3$. Thus $N$ must be one of the five groups listed. 
\end{proof}

\section{Strategy for counting braces and Hopf-Galois structures in the $2$-power case} \label{sect:strategy}

Let $N$ be any of the groups listed in Theorem \ref{thm:types}. In this section, we outline the strategy we will use to count the quaternion and dihedral braces and Hopf-Galois structures corresponding to $N$. We shall then carry out this strategy for each $N$ in turn in the subsequent sections.

Write $N$ as in Proposition \ref{pro:automorphismsmatrix}:
$$ N=\Z/2^{a_1}\Z\times \Z/2^{a_2}\Z\times \cdots \times \Z/2^{a_r}\Z $$
with $1\leq a_1\leq\cdots\leq a_r$, and let $\overline{R}\cong \End(N)$ be as described after Proposition  \ref{pro:automorphismsmatrix}. Then, by Corollary \ref{cor:holomorphmatrix}, we may view an element of $\Hol(N)$ as a matrix 
\begin{equation} \label{eq:holmatrix}
  \begin{pmatrix}
	A       & v  \\
	0       & 1
\end{pmatrix}, 
\end{equation}
with $A$ an invertible element of $\overline{R}$ and with $v \in N$. 

Let $G$ be a quaternion or dihedral regular subgroup of $\Hol(N)$ of order $2^n=|N|$. Then $G$ is generated by elements
\begin{equation} \label{eq:XY}
X= \begin{pmatrix}
	A       & v  \\
	0       & 1
\end{pmatrix}, \quad 
Y= \begin{pmatrix}
	B       & w  \\
	0       & 1
\end{pmatrix}
\end{equation}
of the above form, satisfying the relations $X^{2^{n-1}}=I$, $X^{2^{n-2}}\neq I$, $YX=X^{-1}Y$ and either $Y^2=X^{2^{n-2}}$ or $Y^2=I$ (depending on whether we are considering the quaternion case or dihedral case). For $k \geq 1$, we see inductively that
$$  X^{2^k}= \begin{pmatrix}
	A^{2^k}       & (I+A+ \cdots + A^{2^k-1})v  \\
	0       & 1
\end{pmatrix}.  $$
As $X^{2^{n-2}}\neq I$ and $G$ is regular, we therefore have 
\begin{equation} \label{eq:Xreg}
	(I+A+\cdots+A^{2^{n-2}-1})v \neq  0. 
\end{equation}
Since
$$ Y^{-1} = \begin{pmatrix}
	B^{-1}       & -B^{-1}w  \\
	0       & 1
\end{pmatrix}, $$
the remaining relations become
\begin{equation} \label{eq:XYrels}
\begin{cases}
	A^{2^{n-1}}=I\\
	(I+A+\cdots+A^{2^{n-1}-1})v =  0\\
	BA=A^{-1}B\\
	w+Bv=-A^{-1}v+A^{-1}w\\
	\begin{cases}
		B^2=A^{2^{n-2}}\text{ (quaternion)}\\
		B^2=I\text{ (dihedral)}
	\end{cases}\\
	\begin{cases}
		(I+B)w=(I+A+\cdots+A^{2^{n-2}-1})v\text{ (quaternion)}\\
		(I+B)w=0\text{ (dihedral)}.
	\end{cases}
\end{cases}
\end{equation}
Note that any $X$ and $Y$ as in (\ref{eq:XY}) satisfying (\ref{eq:Xreg}) and (\ref{eq:XYrels}) will generate a subgroup isomorphic to $Q_{2^n}$ or $D_{2^n}$ (except in the degenerate case $n=2$, $Y=I$).

We will determine the regular quaternion and dihedral subgroups in $\Hol(N)$ by finding all solutions of (\ref{eq:Xreg}) and (\ref{eq:XYrels}) and imposing further conditions to ensure regularity. We will often simplify this calculation by working up to conjugacy in $\Hol(N)$. In particular, using Lemma \ref{lem:Sylow}, we may assume that $A$ and $B$ reduce to invertible upper triangular matrices in $\GL_r(\F_2)$. 

In order to avoid complications from small values of $n$, we shall carry out the above strategy assuming that $n \geq 5$ (or $n \geq 4$ for $N=C_{2^n}$). When $n \geq 5$, we shall find that only two of the five groups $N$ of order $2^n$ in Theorem \ref{thm:types} give regular quaternion or dihedral subgroups (cf.~\cite[Theorem 1]{MR4141382}). In \S\ref{sect:small2power} we will treat the omitted small cases $|N|=4$, $8$, $16$ using \textsc{Magma}.

\section{On the number of quaternion and dihedral braces and Hopf--Galois structures of cyclic type}

Let $N=C_{2^n}$, with $n\geq 4$. We apply the strategy described in \S\ref{sect:strategy}. In this case, the matrices $A$ and $B$ in (\ref{eq:XY}) are single elements of $(\Z/2^n\Z)^\times$, say $\alpha$ and $\beta$ respectively, and $v, w \in \Z/2^n \Z$. The condition $BA=A^{-1}B$ in (\ref{eq:XYrels}) gives $\alpha^2=1$. Then (\ref{eq:Xreg}) and the condition $(I+A+ \cdots+A^{2^{n-1}-1})v =0$ in (\ref{eq:XYrels}) give $2^{n-3}(1+\alpha)v \neq 0$ and $2^{n-2}(1+\alpha)v =0$ in $\Z/2^n\Z$. Hence $(1+\alpha)v \equiv 4 \pmod{8}$, and $(\ref{eq:Xreg})$  and (\ref{eq:XYrels}) reduce to 
\[
    \begin{cases}
     (1+\alpha)v\equiv 4\pmod 8\\
     \alpha^2=1\\
     (\alpha+\beta)v=(\alpha-1)w\\
     \beta^2=1\\
      \begin{cases}
      (1+\beta)w=2^{n-1}\text{ (quaternion)}\\
      (1+\beta)w=0\text{ (dihedral)}.
     \end{cases}
    \end{cases}
\]
Note that $\alpha^2\equiv 1\pmod{2^n}$ means that $\alpha$ can only be $1$, $1+2^{n-1}$, $-1$ or $-1+2^{n-1}$ modulo $2^n$. The third and fourth options are not compatible with the first relation, so that $\alpha$ can only be $1$ or $1+2^{n-1}$. 

If $\alpha\equiv 1+2^{n-1}\pmod{2^n}$, then $v\equiv 2\pmod 4$ and by regularity $w$ is odd. Then the last relation tells us that $\beta$ is $-1+2^{n-1}$ in the quaternion case or $-1$ in the dihedral case. Neither of these cases is compatible with the relation $(\alpha+\beta)v=(\alpha-1)w$.

If $\alpha\equiv 1\pmod{2^n}$, then again $v\equiv 2\pmod4$ and $w$ is odd. The last relation tells us that $\beta$ is $-1+2^{n-1}$ in the quaternion case or $-1$ in the dihedral case. We can see that the relation $(\alpha+\beta)v=(\alpha-1)w$ is automatically satisfied. Hence $\alpha$ and $\beta$ are fixed, which means we have at most $2^{n-2}$ choices for $X$. This corresponds to at most one choice for $\langle X\rangle\subseteq \Hol(N)$. In a similar fashion, we have at most $2^{n-1}$ different choices for $Y$. Hence there is at most one possible subset of $\Hol(N)$ of the form $\{X^iY\}_{0\leq i\leq 2^{n-1}-1}$. This implies that there is at most one regular quaternion, respectively dihedral, subgroup. Now note that the subgroups of $\Hol(N)$ generated by
 \[
X=\begin{pmatrix}
    1 & 2 \\
    0 & 1
\end{pmatrix},
Y=\begin{pmatrix}
    -1+2^{n-1} & 1 \\
    0 & 1
\end{pmatrix}
\]
or
 \[
X=\begin{pmatrix}
    1 & 2 \\
    0 & 1
\end{pmatrix},
Y=\begin{pmatrix}
    -1 & 1 \\
    0 & 1
\end{pmatrix}
\]
are regular.

\begin{prop}
 Let $n\geq 4$ be an integer. Then there is one regular quaternion subgroup and one regular dihedral subgroup of $\Hol(C_{2^n})$.
\end{prop}

\begin{corollary} \label{cor:cyc-brace}
 Let $n\geq 4$ be an integer. Then there is only one quaternion brace and one dihedral brace of type $C_{2^n}$.
\end{corollary}

Using Proposition \ref{pro:automorphismsquaterniondihedral} and (\ref{HGS-count}), we finally obtain the following.

\begin{corollary}\label{cor:2powercyclic}
Let $n\geq 4$ be an integer. Any Galois extension of degree $2^n$ with quaternion or dihedral Galois group admits $2^{n-2}$ Hopf--Galois structures of type $C_{2^n}$.
\end{corollary}

\section{On the number of quaternion and dihedral braces and Hopf--Galois structures of type $C_2\times C_{2^{n-1}}$}

We will now assume $n\geq 5$. Let $N=C_2\times C_{2^{n-1}}$. We look for elements $X$, $Y$ in $\Hol(N)$ as in (\ref{eq:XY}). The matrices $A$ and $B$ are in

\[
 \begin{pmatrix}
    \Z/2\Z       & \Z/2\Z  \\
    2^{n-2}\Z/2^{n-1}\Z  & \Z/2^{n-1}\Z
\end{pmatrix}
\]
and are invertible modulo $2$, so they lie in
\[
 \begin{pmatrix}
    1       & \Z/2\Z \\
    2^{n-2}\Z/2^{n-1}\Z & (\Z/2^{n-1}\Z)^\times
\end{pmatrix}.
\]
We will write 
\[
A= \begin{pmatrix}
    1       & a  \\
    2^{n-2}b       & \alpha
\end{pmatrix}, \quad 
B= \begin{pmatrix}
    1       & r  \\
    2^{n-2}s       & \beta
\end{pmatrix}
\]
with $a,b,r,s\in\Z/2\Z$ and $\alpha,\beta\in (\Z/2^{n-1}\Z)^\times$. We also write
$$ v =  \begin{pmatrix}
	v_1  \\
	v_2      
\end{pmatrix},  \quad
 w =  \begin{pmatrix}
	w_1  \\
	w_2      
\end{pmatrix},  $$
with $v_1, w_1 \in \Z/2\Z$ and $v_2, w_2 \in \Z/2^{n-1}\Z$. 


We now express the conditions (\ref{eq:Xreg}) and (\ref{eq:XYrels}) in terms of $a,b,\alpha,r,s,\beta$. First note that
\[
A^{-1}
=\begin{pmatrix}
    1       & a  \\
    2^{n-2}b   & 2^{n-2}ab+\alpha^{-1}
\end{pmatrix}.
\]
so that the condition $BA=A^{-1}B$ is equivalent to
\[
 \alpha-\alpha^{-1}=2^{n-2}(ab+as+br).
\]
This translates into $\alpha^2=1$ or $\alpha^2=1+2^{n-2}$, according as $ab+as+br=0$ or $1$. In particular, note that $A^4=I$. Then the first condition in (\ref{eq:XYrels}) is automatically satisfied and the second condition, together with (\ref{eq:Xreg}), becomes $\alpha\equiv 1\pmod 4$ and $v_2$ odd. In particular $B^2$ has to be the identity in both the quaternion and dihedral case, or equivalently we must have $2^{n-2}rs+\beta^2=1$. Expressing the conditions $w+Bv=-A^{-1}v+A^{-1}w$ and those on $(I+B)w$ explicitly, we now rewrite  (\ref{eq:Xreg}) and (\ref{eq:XYrels})  as:
\begin{equation}\label{eq:conditions1}
    \begin{cases}
     \alpha^2\equiv 1+2^{n-2}(ab+as+br)\pmod{2^{n-1}}\\
     \alpha\equiv 1\pmod 4\\
     v_2\equiv 1\pmod 2\\
     aw_2\equiv (a+r)v_2\pmod 2\\
     \beta\equiv 2^{n-2}(s+b)v_1v_2^{-1}+2^{n-2}ab-\alpha^{-1}+2^{n-2}bw_1v_2^{-1}+(2^{n-2}ab+\alpha^{-1}-1)w_2v_2^{-1}\pmod{2^{n-1}}\\
     rw_2\equiv 0\pmod 2\\
     2^{n-2}rs+\beta^2\equiv 1\pmod{2^{n-1}}\\
      \begin{cases}
      2^{n-2}sw_1+(1+\beta)w_2\equiv  2^{n-2}\pmod{2^{n-1}}\text{ (quaternion)}\\
      2^{n-2}sw_1+(1+\beta)w_2\equiv 0 \pmod{2^{n-1}}\text{ (dihedral)}.
     \end{cases}
    \end{cases}
\end{equation}

Recall that 
$$ X = \begin{pmatrix}
1 & a & v_1 \\
2^{n-2}b & \alpha & v_2 \\
0 & 0 & 1
\end{pmatrix}. $$
We now show that we can make a convenient choice of $v_1$ and $v_2$ by replacing our quaternion or dihedral subgroup by a conjugate. Observe that
\begin{equation}\label{eq:conjugate}
 \begin{pmatrix}
    1 & 1 & 0 \\
    0 & 1 & 0 \\
    0 & 0 & 1
\end{pmatrix}^{v_1}
 \begin{pmatrix}
    1 & 0 & 0 \\
    0 & v_2^{-1} & 0 \\
    0 & 0 & 1
\end{pmatrix}
 \begin{pmatrix}
    1 & a & v_1 \\
    2^{n-2}b & \alpha & v_2 \\
    0 & 0 & 1
\end{pmatrix}
 \begin{pmatrix}
    1 & 0 & 0 \\
    0 & v_2^{-1} & 0 \\
    0 & 0 & 1
\end{pmatrix}^{-1}
 \begin{pmatrix}
    1 & 1 & 0 \\
    0 & 1 & 0 \\
    0 & 0 & 1
\end{pmatrix}^{-v_1}
= 
\end{equation}
\[
 =\begin{pmatrix}
    1 & a & 0 \\
    2^{n-2}b & \alpha+2^{n-2}bv_1 & 1 \\
    0 & 0 & 1
\end{pmatrix}.
\]
This means that, up to conjugacy, we can assume $v_1=0$ and $v_2=1$. 

We now impose further conditions on $a,b,r,s,\alpha,\beta$ to ensure that our subgroup is regular. First note that $X^4\equiv 1 \pmod 4$, so that the entries of $\{X^i(0,0,1)^\top\pmod 4\}_{i=0,\ldots,2^{n-1}-1}$ are 
\[
 \begin{pmatrix}
    0 \\
    0 \\
    1
\end{pmatrix},
 \begin{pmatrix}
    0 \\
    1 \\
    1
\end{pmatrix},
 \begin{pmatrix}
    a \\
    2 \\
    1
\end{pmatrix},
 \begin{pmatrix}
    a \\
    3 \\
    1
\end{pmatrix}.
\]
Since, by regularity, there exist integers $0\leq i\leq 2^{n-1}-1$ and $0\leq j\leq 1$ such that
\[
 X^iY^j
  \begin{pmatrix}
    0 \\
    0 \\
    1
\end{pmatrix}=
 \begin{pmatrix}
    1 \\
    0 \\
    1
\end{pmatrix}
\]
and $j$ cannot be $0$ (else the equality could not hold modulo $4$), we may replace $X$ by $X^iY$ for a suitable $i$ and so assume that  $w_1=1$ and $w_2=0$. (This change of generators does not affect the relations satisfied by $X$ and $Y$.) 

We now rewrite the system (\ref{eq:conditions1}) with the information we have on $v$ and $w$.
\[
    \begin{cases}
     v_1=0\\
     v_2=1\\
     w_1=1\\
     w_2=0\\
     \alpha^2\equiv 1+2^{n-2}as\pmod{2^{n-1}}\\
     \alpha\equiv 1\pmod 4\\
     r=a\\
     \beta\equiv 2^{n-2}b(1+a)-\alpha^{-1}\pmod{2^{n-1}}\\
      \begin{cases}
      s=1\text{ (quaternion)}\\
      s=0\text{ (dihedral)}.
     \end{cases}
    \end{cases}
\]
Conversely, these conditions ensure that the orbit of $0\in N$ under the subgroup $\langle X \rangle$ of $G$ has cardinality $2^{n-1}$ (since they force (\ref{eq:Xreg}) to hold), and our choice of $w$ means that the orbit of $0$ under $G$ is strictly larger than this. As $|G|=|N|=2^n$, it follows that the above conditions are sufficient to guarantee the regularity of $G$. Moreover, they ensure that $Y$ is completely determined by $X$ and the isomorphism class (quaternion or dihedral) of $G$. Hence we have the following result.

\begin{prop}\label{pro:Yisunique}
 Let $\langle X,Y\rangle_o$ be a quaternion or dihedral regular subgroup of $\Hol(N)$. Let $Y'\in\Hol(N)$ be such that $\langle X,Y'\rangle_o$ is also a quaternion, respectively dihedral, regular subgroup of $\Hol(N)$. Then $\langle X,Y\rangle_o=\langle X,Y'\rangle_o$.
\end{prop}

Once we fix $s$ (hence the isomorphism class), $a$ and $b$, there are two possible values of $\alpha$, whose difference is $2^{n-2}$. Then everything else is completely determined. This proves that, up to conjugation, we have $8$ quaternion regular subgroups and $8$ dihedral regular subgroups of $\Hol(N)$: we will call these subgroups the \textit{fundamental subgroups}. As an aside note the following.

\begin{corollary}\label{cor:matrixrealisable}
 Let
  \[
X=\begin{pmatrix}
    1 & a & v_1 \\
    2^{n-2}b & \alpha & v_2 \\
    0 & 0 & 1
\end{pmatrix},
\]
 be a matrix such that $a,b,v_1\in\Z/2\Z$, $v_2\in(\Z/2^{n-1}\Z)^\times$, $\alpha\equiv 1\pmod 4$ and either $\alpha^2\equiv 1\pmod{2^{n-1}}$ or $\alpha^2\equiv 1+2^{n-2}\pmod{2^{n-1}}$. Then there exists $Y\in\Hol(N)$ such that $\langle X,Y\rangle_o$ is a quaternion, respectively dihedral, regular subgroup of $\Hol(N)$.     
\end{corollary}

\begin{proof}
 By (\ref{eq:conjugate}), $X$ can be conjugated via $\Aut(N)$ to a matrix $X'$ such that $\langle X',Y'\rangle_o$ is a fundamental subgroup for some $Y'\in\Hol(N)$, where $a=a'$, $b=b'$ and $\alpha'$ is either $\alpha$ or $\alpha+2^{n-2}$. If $g\in\Hol(N)$ is such that $gXg^{-1}=X'$, then $\langle X,g^{-1}Y'g\rangle_o$ is a regular subgroup of $\Hol(N)$.
\end{proof}

We now show that the fundamental subgroups represent six conjugacy classes for either possibility for $s$. As noted in the introduction, if two regular subgroups are conjugate in $\Hol(N)$ then they are conjugate by an element of $\Aut(N)$. Thus we only need to check for conjugation by elements of the form
\[
 M =  \begin{pmatrix}
    1 & p & 0 \\
    2^{n-2}q & \gamma & 0 \\
    0 & 0 & 1
\end{pmatrix}
\]
with $p,q\in\Z/2\Z$ and $\gamma\in(\Z/2^{n-1}\Z)^\times$. By Lemma \ref{lem:conjugationy}, a necessary condition for $\langle X,Y\rangle_o$ to be conjugate to $\langle X_*,Y_*\rangle_o$ (with parameters denoted by $a_*$, $b_*$ etc) is that, for a certain odd exponent $d$, we have $MXM^{-1}=X_*^d$. Moreover, by (\ref{eq:conjugate}) we may assume $v_{1*}=0$, $v_{2*}=1$. We calculate
\[
 MXM^{-1} =
\begin{pmatrix}
    1 & p & 0 \\
    2^{n-2}q & \gamma & 0 \\
    0 & 0 & 1
\end{pmatrix}
\begin{pmatrix}
    1 & a & 0 \\
    2^{n-2}b & \alpha & 1 \\
    0 & 0 & 1
\end{pmatrix}
   \begin{pmatrix}
    1 & p & 0 \\
    2^{n-2}q & \gamma^{-1}+2^{n-2}pq & 0 \\
    0 & 0 & 1
\end{pmatrix}=
\]
\[
=\begin{pmatrix}
    1 & a & p \\
    2^{n-2}b & \alpha+2^{n-2}(aq+bp) & \gamma \\
    0 & 0 & 1
\end{pmatrix}.
\]
Note that
\[
X_*^{-1}=\begin{pmatrix}
    A_*^{-1}       & -A_*^{-1}v_*  \\
    0       & 1
\end{pmatrix}
=\begin{pmatrix}
    1       & a_*  & a_*\\
    2^{n-2}b_*   & 2^{n-2}a_*b_*+\alpha_*^{-1} & 2^{n-2}a_*b_*-\alpha_*^{-1} \\
    0 & 0 & 1
\end{pmatrix}.
\]
Since $A_*^4=I$, for every positive integer $k$ we have
\[
 X_*^{4k}=
\begin{pmatrix}
    1 & 0 & 0 \\
    0 & 1 & k(1+\alpha_*)(1+\alpha_*^2) \\
    0 & 0 & 1
\end{pmatrix}.
\]
(Recall that $(1+\alpha_*)(1+\alpha_*^2)\equiv 4\pmod 8$.) Thus, depending on $d$ modulo $4$, we find that $X_*^d$ is either of the form
\begin{equation}\label{eq:power1}
X_*^{d}=X_*^{4k+1}=\begin{pmatrix}
    1 & a_* & 0 \\
    2^{n-2}b_* & \alpha_* & k\alpha_*(1+\alpha_*)(1+\alpha_*^2)+1 \\
    0 & 0 & 1
\end{pmatrix}
\end{equation}
or of the form
\begin{equation}\label{eq:power3}
X_*^{d}=X_*^{4k-1}=\begin{pmatrix}
    1 & a_* & a_* \\
    2^{n-2}b_* & 2^{n-2}a_*b_*+\alpha_*^{-1} & k(1+\alpha_*)(1+\alpha_*^2)-\alpha_*^{-1}+2^{n-2}a_*b_* \\
    0 & 0 & 1
\end{pmatrix}.
\end{equation}
Since $MXM^{-1}=X_*^d$, it follows that $a=a_*$ and $b=b_*$ (therefore also $r=r_*$ and, for isomorphism class reasons, $s=s_*$) and either $p=0$ or $p=a$, depending on $d\pmod 4$. Looking at the central entry, this means that either
\[
 \alpha+2^{n-2}aq=\alpha_*,\ \ \text{ or } \ \ \alpha+2^{n-2}aq=\alpha_*^{-1}.
\]
If $a=0$ then the above, together with the fact that in this case $\alpha^2=1$, implies that $\alpha=\alpha_*$. If $a=1$, then we always find a choice for $q$ (and possibly $\gamma$) such that $\alpha_*\neq\alpha$. By Proposition \ref{pro:Yisunique} and Corollary \ref{cor:matrixrealisable}, there is only one possible choice for $Y$ such that $\langle X,Y\rangle_o$ is a quaternion, respectively dihedral, regular subgroup. This means that the conjugations we found when $a=1$ actually pair fundamental groups. We have therefore shown the following result:

\begin{prop} \label{pro:2powertwo-braces}
 Let $n\geq 5$ be an integer. Then there are $6$ quaternion braces and $6$ dihedral braces of type $C_2\times C_{2^{n-1}}$.
\end{prop}

We now count the number of Hopf--Galois structures. We first note that the eight fundamental subgroups are distinct from each other. Suppose that $\langle X_1,Y_1\rangle_o=\langle X_2,Y_2\rangle_o$. We will denote by $a_i$, $b_i$ etc.\ their respective parameters for $i=1,2$. Then, by Lemma \ref{lem:conjugationy}, $\langle X_1\rangle=\langle X_2\rangle$. By (\ref{eq:power1}) and (\ref{eq:power3}), we must have that $a_1=a_2$ and $b_1=b_2$. Note that in $X_1^{4k-1}$ the $(2,3)$-entry is $3$ modulo $4$, so that we cannot have $X_1^{4k-1}=X_2$. If $X_1^{4k+1}=X_2$, then also $\alpha_1=\alpha_2$.

\begin{prop}\label{pro:2powertwo-HGS}
 Let $n\geq 5$ be an integer. Then any Galois extension of degree $2^n$ with quaternion or dihedral Galois group admits $2^{n+1}$ Hopf--Galois structures of type $C_2\times C_{2^{n-1}}$.
\end{prop}

\begin{proof}
 We first count the number of regular subgroups of $\Hol(N)$. Let $\langle X,Y\rangle_o$ be a quaternion or dihedral regular subgroup of $\Hol(N)$. Then $X$ is of the form 
 \[
\begin{pmatrix}
    1 & a & v_1 \\
    2^{n-2}b & \alpha & v_2 \\
    0 & 0 & 1
\end{pmatrix},
\]
where $a,b\in\Z/2\Z$, $v_2\in(\Z/2^{n-1}\Z)^\times$, $\alpha\equiv 1\pmod 4$ and either $\alpha^2\equiv 1\pmod{2^{n-1}}$ or $\alpha^2\equiv 1+2^{n-2}\pmod{2^{n-1}}$, $v_1=0,1$. Note that we have two choices for $a$, two choices for $b$, two choices for $\alpha$, two for $v_1$ and $2^{n-2}$ for $v_2$ (as we proved $v_2$ has to be odd), for a total of $2^{n+2}$. By  Proposition \ref{pro:Yisunique} and Corollary \ref{cor:matrixrealisable}, each such matrix belongs to a unique regular subgroup of $\Hol(N)$. As the number of generators in $\langle X \rangle$ is $2^{n-2}$, this implies that there are $16$ regular subgroups in each isomorphism case.

By Proposition \ref{pro:automorphismsquaterniondihedral}, $G$ has $2^{2n-3}$ automorphisms in both cases. From the matrix description, we also easily see that $N$ has $2^n$ automorphisms. Hence in both cases,  by (\ref{HGS-count}) the number of Hopf--Galois structures is
\[
 \frac{2^{2n-3}}{2^n}\cdot16=2^{n+1}.
\]
\end{proof}

\section{On the number of quaternion and dihedral braces and Hopf--Galois structures of type $C_4\times C_{2^{n-2}}$}

Using \textsc{Magma} (or by a direct calculation) we easily verify the following:
\begin{lemma}\label{lem:4816}
 $\Hol(C_4\times C_8)$ has no elements of order $16$.
\end{lemma}


Lemma \ref{lem:4816} tells us that there are no quaternion or dihedral regular subgroups of $\Hol(C_4\times C_8)$. This implies that we cannot find any brace or Hopf--Galois structure with additive group $C_4\times C_8$ and quaternion or dihedral multiplicative group.

Now let $n\geq 6$. Let $N=C_4\times C_{2^{n-2}}$. Let $G$ be a quaternion or dihedral regular subgroup of $\Hol(N)$. Then $G$ acts transitively on $N$. Note that the $2^{n-5}$-th powers generate a characteristic subgroup $M$ of $N$ such that $N/M\cong C_4\times C_8$, so that we find a transitive action of $G$ on $C_4\times C_8$. The image of $G$ in $\Perm(C_4\times C_8)$ lies in $\Hol(C_4\times C_8)$ as $G$ is a subgroup of $\Hol(N)$. The kernel of the map $f:G\rightarrow \Hol(C_4\times C_8)$ we just obtained has index at least $32$, as the action of $G$ on $C_4\times C_8$ is transitive. Therefore $G/\ker f$, and hence also $\Hol(C_4\times C_8)$, must have an element of order $16$.  This contradicts Lemma \ref{lem:4816}.

\begin{prop} \label{pro:noC4C2}
 Let $n\geq 5$ be an integer. Then there are no quaternion or dihedral braces type $C_4\times C_{2^{n-2}}$. A Galois extension of degree $2^n$ with quaternion or dihedral Galois group does not admit any Hopf-Galois structures of type $C_4\times C_{2^{n-2}}$.
\end{prop}

\section{On the number of quaternion and dihedral braces and Hopf--Galois structures of type $C_2\times C_2\times C_{2^{n-2}}$}\label{sec:C2C2C2C}

Let $n\geq 5$ be an integer, let $N=C_2\times C_2\times C_{2^{n-2}}$, and let $G=\langle X,Y\rangle_o$ be a quaternion or dihedral regular subgroup of $\Hol(N)$. Then, using Corollary \ref{cor:holomorphmatrix}, and conjugating to ensure that $G$ lies in the Sylow $2$-subgroup of $\Hol(N)$ described in Lemma \ref{lem:Sylow}, we can write
\[
 X=\begin{pmatrix}
    1 & a & b & v_1 \\
    0 & 1 & c & v_2\\
    2^{n-3}d & 2^{n-3}e & \alpha & v_3\\
    0 & 0 & 0 & 1
   \end{pmatrix}
\]
where $\alpha\in(\Z/2^{n-2}\Z)^\times$, $a,b,c,d,e,v_1,v_2\in\Z/2\Z$, $v_3\in\Z/2^{n-2}$. 
We now compute $X^2$ and $X^4$:
\[
 X^2=\begin{pmatrix}
    1 & 0 & ac & av_2+bv_3 \\
    0 & 1 & 0 & cv_3\\
    0 & 2^{n-3}ad  & 2^{n-3}(bd+ce)+\alpha^2 & 2^{n-3}dv_1+2^{n-3}ev_2+(1+\alpha)v_3\\
    0 & 0 & 0 & 1
   \end{pmatrix},
\]
\[
 X^4=\begin{pmatrix}
    1 & 0 & 0 & 0 \\
    0 & 1 & 0 & 0\\
    0 & 0  & \alpha^4 & 2^{n-3}acdv_3+(1+\alpha)(1+\alpha^2)v_3\\
    0 & 0 & 0 & 1
   \end{pmatrix}.
\]
Note that $2^{n-3}acdv_3+(1+\alpha)(1+\alpha^2)v_3$ is divisible by $4$ since $n\geq 5$, and that $\alpha^4 \equiv 1 \pmod{4}$. So $X^4 \equiv I \pmod{4}$, and an easy induction shows that $X^{2^k} \equiv I \pmod{2^{k}}$ for $k \geq 2$. Thus $X^{2^{n-2}}=I$ in $\Hol(N)$, contradicting the relations defining $G$. 

\begin{prop} \label{pro:r3}
 Let $n\geq 5$ be an integer. Then there are no quaternion or dihedral braces of type $C_2\times C_2\times C_{2^{n-2}}$.  A Galois extension of degree $2^n$ with quaternion or dihedral Galois group does not admit any Hopf-Galois structures of type $C_2 \times C_2 \times C_{2^{n-2}}$.
\end{prop}

\section{On the number of quaternion and dihedral braces and Hopf--Galois structures of type $C_2\times C_2\times C_2\times C_{2^{n-3}}$}

Let $n\geq 5$ be an integer, let $N=C_2\times C_2\times C_2 \times C_{2^{n-3}}$ and let $\langle X,Y\rangle_o$ be a quaternion or dihedral regular subgroup of $\Hol(N)$. Then, using Corollary \ref{cor:holomorphmatrix} and Lemma \ref{lem:Sylow} as before, we can write
\[
 X=\begin{pmatrix}
    1 & a & b & c & v_1 \\
    0 & 1 & d & e & v_2\\
    0 & 0 & 1 & f & v_3\\
    2^{n-4}g & 2^{n-4}h & 2^{n-4}i & \alpha & v_4\\
    0 & 0 & 0 & 0 & 1
   \end{pmatrix}
\]
where $\alpha\in(\Z/2^{n-3}\Z)^\times$, $a,b,c,d,e,f,g,h,i,v_1,v_2,v_3\in\Z/2\Z$, $v_4\in\Z/2^{n-3}$.
Now note that
\[
 X^2=\begin{pmatrix}
    1 & 0 & ad & ae+bf & av_2+bv_3+cv_4 \\
    0 & 1 & 0 & df & dv_3+ev_4 \\
    0 & 0 & 1 & 0 & fv_4 \\
    0 & 2^{n-4} ga & 2^{n-4}(gb+hd) & \alpha^2+2^{n-4}(gc+he+if) & (1+\alpha)v_4 + 2^{n-4}(gv_1+hv_2 +iv_3)  \\
    0 & 0 & 0 & 0 & 1
   \end{pmatrix}.
\]
and
\[
X^4=\begin{pmatrix}
	1 & 0 & 0 & 0 & adfv_4 \\
	0 & 1 & 0 & 0 & 0 \\
	0 & 0 & 1 & 0 & 0 \\
	0 & 0 & 0 & \alpha^4+2^{n-4}gadf & (1+\alpha)(1+\alpha^2)v_4 + 2^{n-4}k\\
	0 & 0 & 0 & 0 & 1
\end{pmatrix}
\]
for a certain integer $k$. If $n \geq 6$ then $X^8 \equiv I \pmod{8}$ and, analogously 
to \S\ref{sec:C2C2C2C}, we obtain $X^{2^{n-2}}=I$ in $\Hol(N)$. If $n=5$ then $2^{n-3}=4$ and we already have $X^4=I$ in $\Hol(N)$. In both cases, we have a contradiction. 

\begin{prop}  \label{pro:r4}
 Let $n\geq 5$ be an integer. Then there are no quaternion or dihedral braces or Hopf--Galois structures of type $C_2\times C_2\times C_2 \times C_{2^{n-3}}$.  A Galois extension of degree $2^n$ with quaternion or dihedral Galois group does not admit any Hopf-Galois structures of type $C_2\times C_2\times C_2 \times C_{2^{n-3}}$.
\end{prop}

\section{On quaternion and dihedral braces and Hopf--Galois structures of order $4$, $8$ and $16$}\label{sec:magma}
\label{sect:small2power}

For each quaternion or dihedral group $G$ of order $m=4$, $8$, $16$, we will use Theorem \ref{thm:types} to read off all the possible isomorphism classes of $N$. Suppose that in the standard labelling notation we write $N$ as SmallGroup($m$,$j$) and $G$ as SmallGroup($m$,$k$). We will implement the following \textsc{Magma} code.
\begin{verbatim}
 N:=SmallGroup(m,j);
 G:=SmallGroup(m,k);
 H:=Holomorph(N);
 R:=RegularSubgroups(H);
 R;
\end{verbatim}
The last command will list all the conjugacy classes of the regular subgroups of $\Hol(N)$, labelling them and displaying their length. The last label will be the total number of conjugacy classes, which we denote by $l$. Then we run the following.
\begin{verbatim}
 for i in [1..l] do 
    if IsIsomorphic(R[i]`subgroup,G) eq true then
        print i;
    end if;
 end for;
\end{verbatim}
The number of lines in the output will tell us the number of conjugacy classes of regular subgroups of $\Hol(N)$, hence the number $c(j,k)$ of isomorphism classes of braces with additive group $N$ and multiplicative group $G$. The sum of the lengths will count the number $r(j,k)$ of regular subgroups of $\Hol(N)$, which multiplied by $\frac{|\Aut(G)|}{|\Aut(N)|}$ will give the number $h(j,k)$ of Hopf--Galois structures with Galois group $G$ and type $N$. The results are shown in Table \ref{table:codespower2}. The reader can refer to \cite{groupnames} for the `Small Group' labelling as well the automorphism groups of specific groups and other general group properties.

\begin{table}[htb] 
\centering
\caption{Numbers of braces $c(j,k)$, regular subgroups $r(j,k)$ and Hopf-Galois structures $h(j,k)$ for $m=4$, $8$, $16$.}
 \label{table:codespower2}
\begin{tabular}{|c|c||c|c|c|c|c|c|}
\hline
 $N$ & $G$ & $j$ & $k$ & $l$ & $c(j,k)$ & $r(j,k)$ & $h(j,k)$ \\
\hline
 $C_4$ & $C_4$ & 1 & 1 & 2 & 1 & 1 & 1\\

 $C_2\times C_2$ & $C_4$ & 2 & 1 & 2 & 1 & 3 & 1\\ 
\hline
 $C_4$ & $C_2\times C_2$ & 1 & 2 & 2 & 1 & 1 & 3\\

 $C_2\times C_2$ & $C_2\times C_2$ & 2 & 2 & 2 & 1 & 1 & 1 \\ 
\hline
 $C_8$ & $Q_8$ & 1 & 4 & 5 & 1 & 1 & 6 \\

$C_2\times C_4$ & $Q_8$ & 2 & 4 & 14 & 1 & 2 & 6 \\

$C_2\times C_2\times C_2$ & $Q_8$ & 5 & 4 & 8 & 1 & 14 & 2 \\
\hline
 $C_8$ & $D_8$ & 1 & 3 & 5 & 1 & 1 & 2 \\

$C_2\times C_4$ & $D_8$ & 2 & 3 & 14 & 5 & 14 & 14 \\

$C_2\times C_2\times C_2$ & $D_8$ & 5 & 3 & 8 & 2 & 126 & 6 \\
\hline
$C_{16}$ & $Q_{16}$ & 1 & 9 & 8 & 1 & 1 & 4 \\

$C_2\times C_8$ & $Q_{16}$ & 5 & 9 & 66 & 4 & 8 & 16 \\

$C_4\times C_4$ & $Q_{16}$ & 2 & 9 & 83 & 2 & 48 & 16 \\

$C_2\times C_2\times C_4$ & $Q_{16}$ & 10 & 9 & 161 & 1 & 48 & 8  \\

$C_2\times C_2\times C_2\times C_2$ & $Q_{16}$ & 14 & 9 & 39 & 1 & 5040 & 8 \\
\hline
$C_{16}$ & $D_{16}$ & 1 & 7 & 8 & 1 & 1 & 4 \\

$C_2\times C_8$ & $D_{16}$ & 5 & 7 & 66 & 6 & 16 & 32 \\

$C_4\times C_4$ & $D_{16}$ & 2 & 7 & 83 & 0 & 0 & 0 \\

$C_2\times C_2\times C_4$ & $D_{16}$ & 10 & 7 & 161 & 0 & 0 & 0 \\

$C_2\times C_2\times C_2\times C_2$ & $D_{16}$ & 14 & 7 & 39 & 0 & 0 & 0 \\
\hline
\end{tabular}
\end{table}

For $m=4$ and $m=8$, the number of braces $c(j,k)$ and the number of Hopf-Galois structures $h(j,k)$ are already known, but we include them here for completeness. For $m=4$, see \cite[Proposition 2.4]{MR3320237} for braces and \cite{MR1402555,MR1405283} for Hopf-Galois structures. For $m=8$, see \cite[\S5.1]{Kayvan-thesis} for both.

\section{Summary of the $2$-power case}
Recall that $q(2^n)$ denotes the number of isomorphism classes of quaternion braces of order $2^n$. Let $d(2^n)$ denote the number of isomorphism classes of dihedral braces of order $2^n$. 
Summarising the information contained in Corollaries \ref{cor:cyc-brace}, \ref{cor:2powercyclic}, Propositions \ref{pro:2powertwo-braces}, \ref{pro:2powertwo-HGS}, \ref{pro:noC4C2}, \ref{pro:r3}, \ref{pro:r4} and Table \ref{table:codespower2}, we obtain the following values for $q(n)$ and $d(n)$, as well as corresponding totals for the number of Hopf-Galois structures of abelian type on a quaternion or dihedral Galois extension of degree $2^n$. 

\begin{theorem}\label{thm:2power}
For $n\geq 5$, we have $q(2^n)=d(2^n)=7$; also $q(16)=9$,  $q(8)=3$, $q(4)=2$ and $d(16)=7$,  $d(8)=8$, $d(4)=2$.
	
If $n\geq 5$, then there are $9 \cdot 2^{n-2}$ Hopf-Galois structures of abelian type on any Galois extension of degree $2^n$ whose Galois group is quaternion (respectively dihedral).
On a quaternion (respectively, dihedral) extension of degree $16$ there are $52$ (respectively $36$) Hopf-Galois structures of abelian type. On a quaternion (respectively dihedral) extension of degree $8$ there are $14$ (respectively, $22$) Hopf-Galois structures of abelian type. On a cyclic (respectively, elementary abelian) extension of degree $4$ there are $2$ (respectively $4$) Hopf-Galois structures.
\end{theorem}

When $n\geq 5$, the only possible additive groups for a quaternion or dihedral brace of order $2^n$ (or abelian types for a Hopf-Galois structure on a quaternion or dihedral field extension of degree $2^n$) are $C_{2^n}$ and $C_2 \times C_{2^{n-1}}$. The same is true for $n=4$ in the dihedral case, but all $5$ abelian groups of order $16$ may occur in the quaternion case. For $n=3$, all $3$ abelian groups of order $8$ may occur in either case.

\section{The non $2$-power case}\label{sec:non2power}

Let $G$ be a quaternion or dihedral group of order $2^ns$, where $n \geq 2$, $s$ is odd and $s \geq 3$. We recall that this means that either
\[
  G=\langle x,y: x^{2^{n-1}s}=1,yx=x^{-1}y,y^2=x^{2^{n-2}s} \rangle\cong Q_{2^ns}
\]
or
\[
  G=\langle x,y:x^{2^{n-1}s}=1,yx=x^{-1}y,y^2=1 \rangle\cong D_{2^ns}.
\]

The following result is immediate. 
\begin{lemma}\label{lem:oddnormal}
The subgroup $C=\langle x^{2^{n-1}} \rangle$ is the unique subgroup of $G$ of order $s$. 
\end{lemma}
\begin{corollary}
If $G_2$ is any Sylow $2$-subgroup of $G$, we can write $G=C \rtimes G_2$.	 Moreover 
$G_2$ is quaternion or dihedral, respectively.
\end{corollary}
\begin{proof}
The first statement is clear as $C$ is normal in $G$. For the second, note that one Sylow $2$-subgroup is $\langle x^s, y \rangle$. (Recall from Notation \ref{not:def-groups} that we allow the groups of order $4$ as quaternion or dihedral.)
\end{proof}
	
\begin{lemma}\label{lem:commutator}
 The commutator $G'$ of $G$ is generated by $x^2$.
\end{lemma}

\begin{proof}
 We note that $x^2=xyx^{-1}y^{-1}$, so that $x^2$ lies in $G'$. Also $G/\langle x^2\rangle$ is abelian: it is isomorphic to $C_4$ if $n=2$ and $G$ is quaternion, and to $C_2\times C_2$ in all other cases. This shows that $G'=\langle x^2\rangle$.
\end{proof}

Now let $N$ be a finite abelian group of order $2^ns$ and let $G$ be a quaternion or dihedral regular subgroup of $\Hol(N)$. This corresponds to giving the set $N$ an operation $\circ$ such that $(N,+,\circ)$ is a brace with $(N,\circ)\cong G$. Thus we have $G=\{ g_a : a \in N\}$ where $g_a=(a, \lambda_a)$ with $\lambda_a \in \Aut(N)$ given by $\lambda_a(b)=-a+a \circ b$. We have a canonical isomorphism $G \to (N,\circ)$ given by $g_a \mapsto a$, so that $g_a g_b =g_{a \circ b}$ and the action of $G$ on $N$ is given by $g_a \cdot b=a \circ b$.  

For any characteristic subgroup $M$ of $(N,+)$, we have $\lambda_a(M)=M$ for all $a \in N$, so that $M$ is also a subgroup of $(M,\circ)$ and the  brace structure of $N$ restricts to make $(M,+,\circ)$ into a subbrace. In particular, for each prime $p$ dividing $2^ns$, this applies to the (unique) Sylow $p$-subgroup $N_p$ of $N$. It also applies to the unique subgroup $N_s$ of order $s$ (this being the direct product of the $N_p$ over the primes $p$ dividing $s$). 

Let $G_s=\{(a,\lambda_a) : a \in N_s\}$ and let $G_2=\{(b,\lambda_b) : b \in N_2\}$. Thus $G_s$ is the unique subgroup $C$ of order $s$ in $G$, as in Lemma \ref{lem:oddnormal}, and $G_2$ is a specific Sylow $2$-subgroup of $G$ distinguished by the inclusion $G \subseteq \Hol(N)$. The regular action of $G$ on $N$ restricts to a regular action of $G_s$ on $N_s$ (respectively, of $G_2$ on $N_2$), via which we obtain a group operation $\circ_s$ on $N_s$ (respectively $\circ_2$ on $N_2$), which is simply the restriction of $\circ$. We then have canonical isomorphisms $G_s \to (N_s,\circ_s)$ and $G_2 \to (N_2, \circ_2)$.

Our goal is now to reconstruct the brace structure on $N$ from more fundamental components. For the additive group, we have the canonical decomposition $N=N_s \times N_2$, and from now on we will write elements of $N$ in the form $(a,b)$ with $a \in N_s$ and $b \in N_2$. Since $N_s$ and $N_2$ are both characteristic in $N$, we have canonical isomorphisms
$\Aut(N) \to \Aut(N_s) \times \Aut(N_2)$ and $\Hol(N) \to \Hol(N_s) \times \Hol(N_2)$.   
We seek to determine the operation $\circ$ on $N$, or equivalently, the function $\lambda: N \to \Aut(N,+)$ such that
$$ (a,b) \circ (a',b') = (a,b) + \lambda_{(a,b)}(a',b') \mbox { for } a, a' \in N_s \mbox{ and } b,b' \in N_2. $$

We first evaluate $\lambda$ on the summands $N_s$, $N_2$ separately.

\begin{prop}\label{pro:lambda11}
For $a$, $a' \in N_s$ we have $\lambda_{(a,0)}(a',0) = (a',0)$. Thus $(N_s,+,\circ_s)$ is a trivial brace.
\end{prop}
\begin{proof}
For each prime $p$ dividing $s$, let $G_p=\{(b,\lambda_b) : b \in N_p \}$. Then $G_p$ is a Sylow $p$-subgroup of $G$. As $p$ is odd, it follows that $G_p$ is cyclic. Moreover, the image of $G_p$ under the projection $\Hol(N_s) \to \Hol(N_p)$ is a regular subgroup of $\Hol(N_p)$. By \cite[Theorem 4.4]{MR1644203}, $N_p$ must then also be cyclic. Hence $N_s$ is cyclic, and therefore $\Aut(N_s)$ is abelian. By Lemma \ref{lem:commutator}, the canonical homomorphism $G \to \Aut(N) \twoheadrightarrow \Aut(N_s)$ is then trivial on $G_s$. Thus $\lambda_{(a,0)}(a',0) = (a',0)$, and $(a,0) \circ (a',0)=(a+a',0)$.
\end{proof}

\begin{prop} \label{pro:lambda12}
	For $a \in N_s$ and $b' \in N_2$ we have $\lambda_{(a,0)}(0,b') = (0,b')$.
\end{prop}
\begin{proof}
The action of $G$ on $N$ induces a transitive action of $G$ on $N_2$ via the projection $\Hol(N) \to \Hol(N_2)$. 
The stabiliser of any element $b'$ of $N_2$ is a subgroup of $G$ of order $s$, and by Lemma \ref{lem:oddnormal} this can only be $G_s$. Hence for any $a \in N_s$ and $b' \in N_2$, the image of $g_a \cdot b' = (a,0) \circ (0,b')$ under the projection $N \to N_2$ must coincide with $b'$. But $(a,0) \circ (0,b') = (a,0)+\lambda_{(a,0)}(0,b')$, and $\lambda_{(a,0)}(N_2)=N_2$ as $N_2$ is characteristic in $N$. Hence $\lambda_{(a,0)}(0,b') = (0,b')$.
\end{proof}	
 
By Lemma \ref{lem:oddnormal}, the group $G_s$ is normal in $G$. The subgroup $G_2$ is a complement to $G_s$. Hence $(N_s,\circ)$ is normal in $(N,\circ)$ with complement $(N_2,\circ)$, so conjugation in $(N,\circ)$ gives an action $\tau: (N_2, \circ_2) \to \Aut(N_s, \circ_s)$ making $(N,\circ)$ into the semidirect product $(N_s, \circ_s) \rtimes_\tau (N_2, \circ_2)$. We write $\tau_b$ for the image of $b \in N_2$ under $\tau$, so in $(N,\circ)$ we have the relation
\begin{equation} \label{sd-rel}
	  b \circ a = \tau_b(a) \circ b \mbox{ for } a \in N_s \mbox{ and } b \in N_2.
\end{equation}

\begin{prop} \label{pro:lambda21}
	For $b \in N_2$ and $a' \in N_s$ we have $\lambda_{(0,b)}(a',0) = (\tau_b(a'),0)$. 
\end{prop}
\begin{proof}
\begin{eqnarray*}
	(0,b) + \lambda_{(0,b)}(a',0) & = & (0,b) \circ (a',0) \\
	                              & = & (\tau_b(a'),0) \circ (0,b) \\
	                              & = & (\tau_b(a'), b),
\end{eqnarray*}	
where the second equality is a restatement of (\ref{sd-rel}) and the third follows from Proposition \ref{pro:lambda12}.
\end{proof}

The brace $(N_2,+,\circ_2)$ of order $2^n$ is one of those listed in Theorem \ref{thm:2power}, so is not a trivial brace except possibly when $n=2$. Let $\lambda^{(2)}:(N_2,\circ_2) \to \Aut(N_2,+)$ be given by $\lambda^{(2)}_b(b')=-b+b \circ b'$ for $b$, $b' \in N_2$. Then we have 
\begin{equation} \label{eq:lambda22}
	\lambda_{(0,b)}(0,b') = (0,\lambda^{(2)}_b(b')) \mbox{ for } b, b' \in N_2. 
\end{equation}

Putting these pieces together, we can reconstruct the brace $N$ from the (trivial) brace $N_s$, the brace $N_2$ and the homomorphism $\tau: (N_2, \circ) \to \Aut(N_s, \circ)$. In fact, the brace $N$ is the semidirect product of the braces $N_s$ and $N_2$, cf.~\cite[Corollary 2.36]{MR3763907}.

\begin{lemma} \label{lem:sdp}
The operations in the brace $(N,+,\circ)$ are given by 
$$ (a,b) + (a',b') = (a+a', b+b'), $$
$$ (a,b) \circ (a',b') = (a \circ \tau(a'), b \circ_2 b') = (a + \tau(a'), b \circ_2 b') $$
for all $a$, $a' \in N_s$ and $b$, $b' \in N_2$.
\end{lemma}
\begin{proof}
The first equation is immediate from the direct product decomposition $N=N_s \times N_2$ of the additive group.

As for the multiplicative group, we have from Proposition \ref{pro:lambda12} that $(a,b)=(a,0)+(0,b)=(a,0) \circ (0,b)$. Since $\lambda:(N,\circ) \to \Aut(N,+)$ is a group homomorphism, it follows that
$$ \lambda_{(a,b)}(a',b') =  \lambda_{(a,0)} \lambda_{(0,b)}(a',b') =
  \lambda_{(a,0)} \lambda_{(0,b)}(a',0) +  \lambda_{(a,0)} \lambda_{(0,b)}(0,b') $$
where
$$   \lambda_{(a,0)} \lambda_{(0,b)}(a',0) = (\tau_b(a'),0) $$
by Propositions \ref{pro:lambda21} and \ref{pro:lambda11}, and 
$$  \lambda_{(a,0)} \lambda_{(0,b)}(0,b') = (0,\lambda^{(2)}_b(b')) $$
by (\ref{eq:lambda22}) and Proposition \ref{pro:lambda12}. Thus 
\begin{eqnarray*}
	(a,b) \circ (a',b') & = & (a,b) + (\tau_b(a'),0) + (0,\lambda^{(2)}_b(b')) \\
	   & = & (a+\tau_b(a'), b + \lambda^{(2)}(b')) \\
	   & = & (a \circ \tau_b(a'),b \circ_2 b'),
\end{eqnarray*}
as $a \circ \tau_b(a') = a + \tau_b(a')$ by Proposition \ref{pro:lambda11}. This completes the proof.
\end{proof} 

Keeping the abelian group $N=N_s \times N_2$ fixed, we now allow $G$ to vary amongst the regular quaternion (respectively dihedral) subgroups of $\Hol(N)$. We will describe these groups $G$ in terms of the regular quaternion (respectively dihedral) subgroups $H$ of $\Hol(N_2)$. Any such $H$ determines a group operation $\circ_H$ on $N_2$ making $(N_2,+,\circ_H)$ into a quaternion (respectively dihedral) brace. Let $\lambda^H: (N_2,\circ_H) \to \Aut(N_s)$ be the corresponding homomorphism, so $\lambda^H_b(b')=-b+b \circ_H b'$ for $b$, $b' \in N_2$. 
We denote by $T_H$ the set of homomorphisms $\tau:(N_2,\circ_H) \to \Aut(N_s)$ for which $N_s \rtimes_\tau (N_2,\circ_H)$ is a quaternion (respectively  dihedral) group.    
 
There is an action of the group $\Aut(N_2)$ on the set of all functions $\sigma:N_2 \to \Aut(N_s)$ by $(\beta \cdot \sigma)_b=\sigma_{\beta^{-1}(b)}$ for $\beta \in \Aut(N_2)$ and $b \in N_2$. 

\begin{theorem}\label{thm:correspondence}
Let $n \geq 2$, let $s\geq 3$ be odd, and let $N=N_s\times N_2$ be a finite abelian group of order $2^n s$, where $N_s$ is cyclic of order $s$. Then there is a bijective correspondence between regular quaternion (respectively, dihedral) subgroups $G$ in $\Hol(N)$ and pairs $(H, \tau)$ where $H$ is a regular quaternion (respectively, dihedral) subgroup of $\Hol(N_2)$ and $\tau \in T_H$.
Moreover, if $G$ corresponds to the pair $(H,\tau)$, then, for $\alpha \in \Aut(N_s)$, $\beta \in \Aut(N_2)$, the group
$(\alpha,\beta)G(\alpha,\beta)^{-1}$ corresponds to the pair $(\beta H \beta^{-1}, \beta \cdot \tau)$.  
\end{theorem}

\begin{proof}
Let $G$ be a regular quaternion (respectively, dihedral) subgroup of $\Hol(N)$. As above, $G$ determines a group operation $\circ$ on $N$ and also determines subgroups $G_s$, $G_2$ of $\Hol(N)$. Let $H$ be the image of $G_2$ under the projection $\Hol(N)=\Hol(N_s) \times \Hol(N_2) \to \Hol(N_2)$. Then $H$ is a regular quaternion (respectively, dihedral) subgroup of $\Hol(N_2)$, and the group operation $\circ_H$ on $N_2$ determined by $H$ is just the restriction $\circ_2$ of $\circ$ to $N_2$. Moreover, by Proposition \ref{pro:lambda11}, the restriction of $\circ$ to $N_s$ makes $N_s$ into a trivial brace, and, as explained just before Proposition \ref{pro:lambda21}, there is a homomorphism $\tau:(N_2,\circ_2) \to \Aut(N_s,\circ)=\Aut(N_s,+)$ so that $N_s \rtimes_\tau (N_2, \circ_2) = (N,\circ) \cong G$. Since $\circ_2$ coincides with $\circ_H$, we have $\tau \in T_H$.  
We associate the pair $(H,\tau)$ to $G$. Conversely, let $(H,\tau)$ be a pair as in the statement of the theorem. Then $H$ gives rise to a group operation $\circ_H$ on $N_2$ and a canonical isomorphism $(N_2, \circ_H) \cong H$. Using $\circ_H$ and $\tau$, we define an operation $\circ$ on $N$ by 
$$ (a,b) \circ (a',b') = (a + \tau(a'), b \circ_H b') $$
for all $a$, $a' \in N_s$ and $b$, $b' \in N_2$. It is straightforward to check that this makes $N$ into a quaternion (respectively, dihedral) brace which therefore corresponds to a regular quaternion (respectively, dihedral) subgroup $G$ of $\Hol(N)$. It is clear that these two constructions are mutually inverse. This establishes the bijective correspondence in the theorem.

Now let $(H,\tau)$ correspond to $G$.
We will write elements of $\Hol(N)=(N_s \rtimes \Aut(N_s)) \times \Hol(N_2)$ as triples $(a,\alpha,h)$ with $a \in N_s$, $\alpha \in \Aut(N_s)$, $h \in \Hol(N_2)$. Then 
$$  G=\{(a,\tau_b, h_b) : a \in N_s, b \in N_2\} $$
where $h_b = (b, \lambda^H_b) \in H$. For $\alpha \in \Aut(N_s)$ and $\beta \in \Aut(N_2) \subset \Hol(N_2)$, we have 
$$ (0,\alpha,\beta) (a, \tau_b, h_b) (0, \alpha, \beta)^{-1} = (\alpha(a), \tau_b, \beta h_b \beta^{-1}), $$
where we have again used the fact that $\Aut(N_s)$ is abelian. 
Moreover in $\Hol(N_2)$ we have
$$ \beta h_b \beta^{-1} = (0,\beta)(b,\lambda^H_b)(0,\beta^{-1}) = (\beta(b), \beta \lambda^H_b \beta^{-1}). $$
Writing $H'=\beta H \beta^{-1} =\{h'_b : b \in N_2\}$ with $h'_b=(b,\lambda^{H'}_b)$, 
we therefore have $\beta h_b \beta^{-1} = h'_{\beta(b)}$. Hence
\begin{eqnarray*}
	(\alpha,\beta) G (\alpha,\beta)^{-1} 
	& = & \{ (\alpha(a), \tau_b, h'_{\beta(b)} ) : a \in N_s, b \in N_2 \}  \\
	& = & \{ (a, (\beta \cdot \tau)_{\beta(b)}, h'_{\beta(b)} ) : a \in N_s, b \in N_2 \}  \\	
	& = & \{ (a, (\beta \cdot \tau)_b, h'_b ) : a \in N_s, b \in N_2 \}  	
\end{eqnarray*}
which is the regular subgroup of $\Hol(N)$ corresponding to the pair $(\beta H \beta^{-1}, \beta \cdot \tau)$. 
\end{proof}

\begin{remark}  \label{rem:CGRV}
 The proof of Theorem \ref{thm:correspondence} is inspired by \cite[Proposition 4]{MR4559373}. Note in fact that their proposition holds under the general assumptions (satisfied in both cases) that the braces we are studying are semidirect products of braces such that the normal component is a trivial brace.
\end{remark}

We next investigate the sets $T_H$.

\begin{prop} \label{pro:TH}
Let $H$ be a regular quaternion (respectively dihedral) subgroup of $\Hol(N_2)$. If $H \cong Q_{2^n}$ with $n \neq 3$, or $H \cong D_{2^n}$ with $n \neq 2$, then $|T_H|=1$. If $H \cong Q_8$ or $D_4$ then $|T_H|=3$. In particular, $T_H$ does not depend on $s$.
\end{prop}
\begin{proof}
A homomorphism $\tau:(N_2,\circ_H) \to \Aut(N_s)$ belongs to $T_H$ if and only if    	
$N_s \rtimes_\tau (N_2,\circ_H)$ is a quaternion (respectively dihedral) group. This occurs precisely when the image of $\tau$ is the subgroup of order $2$ in $\Aut(N_s)$ generated by inversion. Thus $\tau$ is completely determined by its kernel, which may be any subgroup of index $2$ in $(N_2,\circ_H) \cong H$. If $H \cong Q_{2^n}$ with $n \neq 3$, or $H \cong D_{2^n}$ with $n \neq 2$, then $H$ has a unique subgroup of index $2$. Thus  $|T_H|=1$. In the remaining cases, $H$ has $3$ subgroups of index $2$, so $|T_H|=3$.
\end{proof}

We now modify the notation used in Table \ref{table:codespower2} to remove the dependence on the `Small Group' labelling. For an abelian group $N$ and a group $J$ with $|N|=|J|=2^n s$, let $c(N,J)$ denote the number of isomorphism classes of braces with additive group isomorphic to $N$ and with multiplicative group isomorphic to $J$. This is just the number of orbits of regular subgroups in $\Hol(N)$ isomorphic to $J$ under conjugation by $\Aut(N)$. Similarly, let $r(N,J)$ be the total number of regular subgroups in $\Hol(N)$ isomorphic to $J$, and let $h(N,J)$ denote the number of Hopf-Galois structures of type $N$ on a Galois extension whose Galois group is isomorphic to $J$. 

The next result shows that, when $J$ is a quaternion or dihedral group, the quantities $r(N,J)$ and $c(N,J)$ almost always coincide with the corresponding quantities for the Sylow $2$-subgroups of $N$ and $J$.

\begin{theorem}    \label{thm:reduction2power}
Let $N=N_s \times N_2$ be as in Theorem \ref{thm:correspondence}, let $J \cong Q_{2^n s}$ or $D_{2^n s}$, and let $J_2$ be a Sylow $2$-subgroup of $J$.

If $J \cong Q_{2^n s}$ with $n \neq 3$, or $J \cong D_{2^n s}$ with $n \neq 2$, then 
$r(N,J)=r(N_2,J_2)$ and $c(N,J)=c(N_2,J_2)$.

If $J \cong Q_{8s}$ or $D_{4s}$ then
$r(N,J) = 3 r(N_2,J_2)$. Also for $n=3$ we have $c(N,Q_{8s})=c(C_3 \times N_2,Q_{24})$, and for $n=2$ we have $c(N,D_{4s}) = c(C_3 \times N_2, D_{12})$.	
\end{theorem}
\begin{proof}
If $J \cong Q_{2^n s}$ with $n \neq 3$, or $J \cong D_{2^n s}$ with $n \neq 2$, then $|T_H|=1$ by Proposition \ref{pro:TH}. Thus Theorem \ref{thm:correspondence} gives a bijection between regular quaternion (respectively dihedral) subgroups $G$ of $\Hol(N)$ and regular quaternion (respectively dihedral) subgroups $H$ of $\Hol(N_2)$. Hence $r(N,J)=r(N_2,J_2)$. Writing
$S_H = \{ \beta \in \Aut(N_2) : \beta H \beta^{-1} = H \}$
for the stabiliser of $H$ in $\Aut(N_2)$, we see that the stabiliser of $G$ in $\Aut(N)=\Aut(N_s) \times \Aut(N_2)$ is $\Aut(N_s) \times S_H$. Thus the length of the conjugacy class of $G$ in $\Aut(N_2)$ is 
$$ \frac{|\Aut(N_s) \times \Aut(N_2)|}{|\Aut(N_s) \times S_H|} = \frac{|\Aut(N_2)|}{|S_H|}, $$
which coincides with the length of the conjugacy class of $H$ in $\Aut(N_2)$. Thus we also have 
$c(N,J)=c(N_2,J_2)$.

In the two remaining cases, we have $|T_H|=3$ for each $H$. Thus each regular quaternion (respectively dihedral) subgroup $H$ in $\Hol(N)$ occurs in $3$ pairs $(H,\tau)$ in Theorem \ref{thm:correspondence}, so $r(N,J)= 3r(N_2,J_2)$. We consider the conjugacy class of the regular subgroup $G$ of $\Hol(N)$  corresponding to the pair $(H,\tau)$. Now $(\alpha,\beta) \in \Aut(N)$ is in the stabiliser of $G$ if and only if $\beta H \beta^{-1}$ and $\beta \cdot \tau = \tau$ (with no condition on $\alpha$). Then $\beta$ must belong firstly to $S_H$, and further to the stabiliser $S_{H,\tau}$ of $\tau$ in $S_H$. Here the action of $S_H$ on $T_H$ corresponds to its action on the subgroups of index $2$ in $H$. Thus $S_{H,\tau}$ may depend on the subgroup $H$, but is independent of $s$. The lengths of the conjugacy classes of the groups $G$ are therefore the same for arbitrary odd $s \geq 3$ as for $s=3$. Hence $c(N,J)$ is as stated. 
\end{proof}

To complete the evaluation of $c(N,J)$ in all cases, it only remains to handle $J \cong Q_{24}$ and $J \cong D_{12}$. We do this using \textsc{Magma}. We start with the case $N=C_{24}$ and $J=Q_{24}$. As in \S\ref{sec:magma}, we run the following code. We omit the output of `R;' which we use to count the number of conjugacy classes of regular subgroups. In this case, there are $14$ conjugacy classes.
\begin{verbatim}
 N:=SmallGroup(24,2);
 G:=SmallGroup(24,4);
 H:=Holomorph(N);
 R:=RegularSubgroups(H);
 R;
 for i in [1..14] do 
    if IsIsomorphic(R[i]`subgroup,G) eq true then
        print i;
    end if;
 end for;
 6
 14
\end{verbatim}
There are therefore two conjugacy classes of regular subgroups isomorphic to $Q_{24}$. Hence there are two isomorphism classes of braces of order $24$ with cyclic additive group and quaternion multiplicative group. 
Repeating the calculation with $N=C_3\times C_2\times C_4$ (SmallGroup(24,9)) and $N=C_3\times C_2^3$ (SmallGroup(24,15)), we find three conjugacy classes 
and one conjugacy class 
of regular subgroups isomorphic to $Q_{24}$, respectively.
\begin{corollary}\label{cor:bracesQ8s}
 Let $s\geq 3$ be an odd number. Then there are $6$ isomorphism classes of braces with multiplicative group isomorphic to $Q_{8s}$.
\end{corollary}

Analogously, we can analyse the $D_{12}$ case, starting from $N=C_{12}$.  
\begin{verbatim}
 N:=SmallGroup(12,2);
 G:=SmallGroup(12,4);
 H:=Holomorph(N);
 R:=RegularSubgroups(H);
 R;
 for i in [1..5] do 
    if IsIsomorphic(R[i]`subgroup,G) eq true then
        print i;
    end if;
 end for;
 2
 5
\end{verbatim}
In this case there are two conjugacy classes.
Hence there are two isomorphism classes of braces with cyclic additive group and quaternion multiplicative group. Repeating the calculation with $N=C_3\times C_2\times C_2$ (SmallGroup(12,5)), we find one conjugacy class 
of regular subgroups of $\Hol(N)$ isomorphic to $D_{12}$.

\begin{corollary}\label{cor:bracesD4s}
 Let $s\geq 3$ be an odd number. Then there are $3$ isomorphism classes of braces with multiplicative group isomorphic to $D_{4s}$.
\end{corollary}

\section{Final statement about braces}
Putting together Theorem \ref{thm:2power} and the conclusions in \S\ref{sec:non2power}, we have now proved Conjecture 6.6 of \cite{MR3647970} which we stated as Conjecture \ref{quat-conj}.

\begin{theorem} \label{thm:GV-conj}
 Let $m\geq 3$ be an integer and let $q(4m)$ be the number of isomorphism classes of braces with multiplicative group isomorphic to $Q_{4m}$. Then 
 \[
  q(4m)=\begin{cases}
         2\ \ \text{ if $m$ is odd;} \\
         6\ \ \text{ if } m \equiv 2 \pmod{4}; \\
         9\ \ \text{ if } m \equiv 4 \pmod{8}; \\
         7\ \ \text{ if } m \equiv 0 \pmod{8}.
        \end{cases}
 \]
\end{theorem}

We also have the analogue of Theorem \ref{thm:GV-conj} for dihedral braces.

\begin{theorem}
	Let $m\geq 3$ be an integer and let $d(4m)$ be the number of isomorphism classes of braces with multiplicative group isomorphic to $D_{4m}$. Then 
	\[
	d(4m)=\begin{cases}
		3\ \ \text{ if $m$ is odd;}\\
		8\ \ \text{ if } m \equiv 2 \pmod{4}; \\
		7\ \ \text{ if } m \equiv 4 \pmod{8}; \\
		7\ \ \text{ if } m \equiv 0 \pmod{8}.
	\end{cases}
	\]
	
\end{theorem}

\begin{table}[htb] 
	\centering
	\caption{Number of isomorphism classes of quaternion or dihedral braces for each possible additive group $N$}  \label{table:abgps}
	\begin{tabular}{|c|c|c|c|}
		\hline
		$N$ & Conditions & Quaternion braces & Dihedral braces \\ \hline
		$C_s \times C_{2^n}$ & $n\geq 5$, $s$ odd & 1 & 1 \\
		$C_s \times C_2 \times C_{2^{n-1}}$ & $n \geq 5$, $s$ odd & 6 & 6 \\ \hline		
		$C_s \times C_{16}$ & $s$ odd & 1 & 1 \\
		$C_s \times C_2 \times C_8$ & $s$ odd & 4 & 6 \\
		$C_s \times C_4 \times C_4$ & $s$ odd & 2 & 0 \\
		$C_s \times C_2 \times C_2 \times C_4$ & $s$ odd & 1& 0 \\
		$C_s \times C_2 \times C_2 \times C_2 \times C_2$  & $s$ odd & 1 & 0 \\ \hline
		$C_s \times C_8$ & $s \geq 3$ odd & 2 & 1 \\
		$C_s \times C_2 \times C_4$ & $s \geq 3$ odd & 3 & 5 \\
		$C_s \times C_2 \times C_2 \times C_2$ & $s \geq 3$ odd & 1 & 2 \\\hline
		$C_8$ &  & 1 & 1 \\
		$C_4 \times C_2 $ & & 1 & 5 \\
		$C_2 \times C_2 \times C_2$ &  & 1 & 2 \\\hline
		$C_s \times C_4$ & $s \geq 3$ odd &  1  & 2  \\
		$C_s \times C_2 \times C_2$ & $s \geq 3$ odd &  1  &  1 \\ \hline
		$C_4$ &  &  1  &  1 \\
		$C_2 \times C_2$ & &  1  &  1  \\ 
		\hline
	\end{tabular}
\end{table}

Question 6.7 of \cite{MR3647970} asks which finite abelian groups $A$ appear as the additive group of a quaternion brace, and Question 6.8 asks for the number of isomorphism classes of quaternion braces for each such $A$. The answers to these questions, and the corresponding answers for dihedral braces, are given in Table \ref{table:abgps}.

\section{Final statement about Hopf--Galois structures}

We start by generalising Proposition \ref{pro:automorphismsquaterniondihedral}.

\begin{prop}\label{pro:automorphismsquaterniondihedralgeneral}
 Let $s\geq 3$ be an odd integer and let $n\geq 2$ be an integer. Then $|\Aut(Q_{2^ns})|=|\Aut(D_{2^ns})|=2^{2n-3}s\varphi (s)$.
\end{prop}

\begin{proof}
 Let $G=\langle x,y\rangle_o$ be $Q_{2^ns}$ or $D_{2^ns}$. The cyclic group $\langle x\rangle$ of order $2^{n-1}s$ is characteristic in $G$, so that every automorphism can only send $x$ to $\varphi(2^{n-1}s)$ other possible elements. Moreover, $y$ can only go to a power of $x$ times $y$. We can verify that each of these possibilities defines an automorphism.
\end{proof}

Using the notation $h(N,J)$ introduced before Theorem \ref{thm:reduction2power}, we can now reduce the enumeration of Hopf-Galois structures on quaternion or dihedral extensions to the $2$-power case.

\begin{prop}  \label{pro:HGSreduction}
Let $s \geq 3$ be odd and let $n\geq 2$. Let $J \cong Q_{2^n s}$ or $J \cong D_{2^n s}$, let $J_2$ be a Sylow $2$-subgroup of $J$, and let $N=N_s \times N_2$ be an abelian group of order $2^n s$ with $N_s$ cyclic of order $s$. Then $h(N,J) = h(N_2,J_2) s$ in all cases.
\end{prop}
\begin{proof}
By (\ref{HGS-count}) we have
$$ h(N,J) = \frac{|\Aut(J)|}{|\Aut(N)|} \; r(N,J) \mbox{ and } 
       h(N_2,J_2) = \frac{|\Aut(J_2)|}{|\Aut(N_2)|} \; r(N_2,J_2). $$ 
As $\Aut(N)=\Aut(N_s) \times \Aut(N_2)$ and $|\Aut(N_s)| = \varphi(s)$, we deduce that 
\begin{equation} \label{eqn:hjn}
	 h(N,J) = \frac{|\Aut(J)|}{|\Aut(J_2)|} \cdot \frac{1}{{\varphi(s)}} \cdot \frac{r(N,J)}{r(N_2,J_2)} \cdot h(N_2,J_2).  
\end{equation}
If $J \cong Q_{2^n s}$ with $n \neq 3$, or $J \cong D_{2^n s}$ with $n \neq 2$, then from Propositions \ref{pro:automorphismsquaterniondihedral} and 
\ref{pro:automorphismsquaterniondihedralgeneral} we have  $|\Aut(J)|=s \varphi(s) |\Aut(J_2)|$. Also 
$r(N,J)=r(N_2,J_2)$ by Theorem \ref{thm:reduction2power}. 
Then (\ref{eqn:hjn}) gives $h(N,J)=h(N_2,J_2) s$.

In the remaining cases $J\cong Q_{8s}$ and $J \cong D_{4s}$, we have $|\Aut(J_2)|=3 \cdot 2^{2n-3}$. (The extra factor $3$ arises since there are $3$ subgroups of index $2$ in $Q_8$ and $D_4$.) Thus $|\Aut(J)|=\frac{1}{3} s \varphi(s) |\Aut(J_2)|$. But by Theorem \ref{thm:reduction2power}, $r(N,J)= 3 r(N_2,J_2)$ in these cases, so again (\ref{eqn:hjn}) yields $h(N,J)=h(N_2,J_2) s$. 
\end{proof}

Combining Proposition \ref{pro:HGSreduction} and Table \ref{table:codespower2}, we obtain our final result. 

\begin{theorem}
Let $n\geq 2$ be an integer and let $s\geq 3$ be an odd integer. 
If $n\geq 5$, then there are $2^{n-2}\cdot 9s$ Hopf-Galois structures of abelian type on a Galois extension of degree $2^ns$ with quaternion or dihedral Galois group. There are $2s$, respectively $14s$, respectively $52s$, Hopf-Galois structures of abelian type on a quaternion Galois extension of degree $4s$, respectively $8s$, respectively $16s$.  
There are $4s$, respectively $22s$, respectively $36s$, Hopf-Galois structures of abelian type on a dihedral Galois extension of degree $4s$, respectively $8s$, respectively $16s$. The numbers of Hopf-Galois structures of each possible abelian type are as shown in Table \ref{table:HGS}.  
\end{theorem}

\begin{table}[htb] 
	\centering
	\caption{Number of Hopf-Galois structures of each possible abelian type $N$ on a Galois extension of degree $2^n s$ ($s \geq 1$, odd) with quaterrnionic or dihedral Galois group $G$}  \label{table:HGS}
	\begin{tabular}{|c|c|c|c|}
		\hline
		$N$ & Conditions & $G$ quaternion & $G$ dihedral \\ \hline
		$C_s \times C_{2^n}$ & $n\geq 5$ & $2^{n-2} s$ & $2^{n-2} s$ \\
		$C_s \times C_2 \times C_{2^{n-1}}$ & $n \geq 5$ & $2^{n+1} s$ & $2^{n+1} s$ \\ \hline		
		$C_s \times C_{16}$ & & $4s$ & $4s$ \\
		$C_s \times C_2 \times C_8$ & & $16s$ & $32s$ \\
		$C_s \times C_4 \times C_4$ & & $16s$ & 0 \\
		$C_s \times C_2 \times C_2 \times C_4$ &  &  $8s$ & 0 \\
		$C_s \times C_2 \times C_2 \times C_2 \times C_2$  & & $8s$ & $0$ \\ \hline
		$C_s \times C_8$ & & $6s$ & $2s$ \\
		$C_s \times C_2 \times C_4$ & & $6s$ & $14s$ \\
		$C_s \times C_2 \times C_2 \times C_2$ &  & $2s$ & $6s$ \\\hline
		$C_s \times C_4$ & &  $s$  & $3s$  \\
		$C_s \times C_2 \times C_2$ & &  $s$  &  $s$ \\ \hline
		\hline
	\end{tabular}
\end{table}

\bibliography{Quaternion-Dihedral-bib.bib}{}
\bibliographystyle{amsalpha}

\end{document}